\documentclass[reqno,12pt, oneside]{amsart}

\usepackage{enumerate}
\usepackage[nobysame]{amsrefs}
\usepackage{enumitem}
\usepackage{amssymb}
\usepackage{amsmath}
\usepackage{amsthm}
\usepackage{mathtools}
\usepackage{amscd}
\usepackage{microtype}
\usepackage[right=2cm,left=2cm, top=2.5cm, bottom=2.5cm,includehead]{geometry}
\usepackage{amsfonts}

\usepackage{tgpagella}
\usepackage{eulervm}

\usepackage[backgroundcolor=lightgray]{todonotes}

\theoremstyle{plain}
\newtheorem{theorem}{Theorem}
\newtheorem{corollary}[theorem]{Corollary}

\newtheorem{proposition}[theorem]{Proposition}

\theoremstyle{remark}
\newtheorem{remark}[theorem]{Remark}
\newtheorem{example}[theorem]{Example}

\theoremstyle{definition}

\DeclarePairedDelimiter{\abs}{\lvert}{\rvert}

\DeclarePairedDelimiter{\norm}{\lVert}{\rVert}

\DeclarePairedDelimiterX{\dset}[2]{\lbrace}{\rbrace}{#1\;\delimsize|\;#2}

\newcommand{\ball}{{\overline{B}}}

\DeclareMathOperator{\card}{card}

\title[When surjectivity demands pathology]{When surjectivity demands pathology}

\author{Piotr Kasprzak}
\address[P. Kasprzak]{Department of Nonlinear Analysis and Applied Topology\\
Faculty of Mathematics and Computer Science\\
  Adam Mickiewicz University in Pozna\'n\\
  ul.\ Uniwersytetu Pozna\'nskiego 4\\
  61-614 Pozna\'n\\
  Poland}
\email[P.~Kasprzak]{kasp@amu.edu.pl}

\author{Piotr Ma\'ckowiak}
\address[P. Ma\'ckowiak]{Department of Nonlinear Analysis and Applied Topology\\
Faculty of Mathematics and Computer Science\\
  Adam Mickiewicz University in Pozna\'n\\
  ul.\ Uniwersytetu Pozna\'nskiego 4\\
  61-614 Pozna\'n\\
  Poland}
\email[P.~Ma\'ckowiak]{piotr.mackowiak@amu.edu.pl}


\date{\today}

\begin{document}
\begin{abstract}
In this paper, we characterize surjective nonlinear composition operators acting between the space of sequences of bounded variation and the space of bounded sequences. We show that their generators exhibit highly pathological behavior. To complement this, we construct explicit examples, confirming that such exotic objects do in fact exist.
\end{abstract}

\maketitle

\section{Introduction}

Consider a discrete dynamical system evolving in time, where the state is updated at each step. A basic example is a (random) walk in Euclidean space, where $x(n) \in \mathbb R^3$ denotes the position of a particle at time $n \in \mathbb N$. Such systems are often studied through path-dependent or energy-like quantities that governs their evolution. (For an introduction to random walks, see the classical monograph by Spitzer~\cite{FS}, and for a deep connection between probability theory (random walks) and electrical network theory, see~\cite{DS}.) 

A basic principle in classical mechanics is that the physical properties of an isolated system should not depend on the particular choice of reference point. Thus, rather than the absolute position of a particle, it is often the changes in its position that are physically relevant. From this perspective, it is natural to consider quantities that depend on increments of the state, with no dependence on the state itself. This leads naturally to expressions of the form 
\[
\sum_{n=1}^\infty\norm{x(n+1)-x(n)}
\]
rather than the classical $l^1$-quantity
\[
 \sum_{n=1}^\infty\norm{x(n)}.
\]
The space of sequences of bounded variation thus arises as a natural functional framework in which such systems may be studied (for the details see Section~\ref{sec:bv}).

This perspective becomes even more natural when one moves beyond Euclidean settings. In many applications, the state at each time step is not a finite-dimensional vector, but a function or another infinite-dimensional object. This occurs, for instance, in discrete analogues of the heat equation, or other important PDEs from engineering, where $x(n)$ represents the spatial heat profile at time $n$ (for more detail see, for example,~\cite{LangtangenLogg}*{Chapter~3}). In such situations, the general setting of Banach spaces provides a more convenient framework.

Now suppose that the system is not observed directly, and that the observer does not have access to the underlying sequence itself. Instead, the available measurements are obtained by applying some transformation to the state at each time step, resulting in a bounded observed signal. This leads to an inverse problem: when does an observed sequence arise from an underlying physical process through a given observation rule, and what structural properties of that rule ensure that such a reconstruction is possible?

Mathematically, this reduces to the study of a nonlinear composition operator between the space of sequences of bounded variation and the space of bounded sequences; surjectivity corresponds to the existence of a reconstruction, whereas bijectivity ensures its uniqueness.

As shown in Section~2, the generators of surjective composition operators of this type necessarily exhibit highly pathological behaviour. Nevertheless, Section~3 demonstrates that such generators do exist, with their construction relying crucially on the fact that, in infinite-dimensional settings, a Banach space and its Hamel basis have the same cardinality.

The paper sits naturally within ongoing research on superposition operators in various BV-type spaces (see, for example,~\cites{ABM, AZ, reinwand} and the references given therein).

\section{Preliminaries}

In this short section, we introduce the notation used throughout the paper and review basic facts concerning the space of sequences of bounded variation.

\subsection{Notation}

As usual, $\mathbb N$ denotes the set of positive integers and $\mathbb R$ is the set of real numbers. Following the standard convention, we write $A\subseteq B$ if the set $A$ is included in $B$, and $A\subset B$ if the inclusion is strict. Throughout the paper, $E$ denotes a real Banach space. For $u \in E$ and $r>0$ we write $\ball_E(u,r)$ and $B_E(u,r)$ for the closed and open balls in $E$, respectively, centred at $u$ and of radius $r$. Moreover, when $r=0$, we define $\ball_E(u,0):=\{u\}$. Finally, a set is said to be countable if it is in bijection with a subset of $\mathbb N$; in particular, the empty set is countable.

\subsection{Space of sequences of bounded variation}
\label{sec:bv}

A sequence $x \colon \mathbb N \to E$ is said to have \emph{bounded variation} if $\sum_{n=1}^\infty \norm{x(n+1)-x(n)}<+\infty$. The space of all such sequences is denoted by $bv(E)$. It is well-known that $bv(E)$ is a Banach space, when endowed with the norm
\[
 \norm{x}_{bv}:=\norm{x(1)}+\sum_{n=1}^\infty \norm{x(n+1)-x(n)}.
\]
Moreover, $bv(E)\subset c(E) \subset l^\infty(E)$, where $c(E)$ and $l^\infty(E)$ stand for the Banach spaces of all convergent and bounded sequences, respectively. Both are endowed with the supremum norm $\norm{x}_\infty :=\sup_{n \in \mathbb N} \norm{x(n)}$. Further information on the relationship between $bv(E)$ and other classical sequence spaces may be found, for example, in \cite{BK}*{Section~ 2.2} and the references therein.

We distinguish between sequences viewed as elements of $bv(E)$ or $l^\infty(E)$ and sequences of elements of the underlying space $E$. The former will be written in functional notation $n \mapsto x(n)$, whereas the latter will be denoted in the standard way by $(u_n)_{n \in \mathbb N}$.

\subsection{Composition operators}

Let $E^{\mathbb N}$ denote the linear space of all $E$-valued sequences. The (\emph{nonlinear}) \emph{composition operator} (also known as the \emph{autonomous superposition operator} or \emph{Nemytskii operator}) $C_f \colon  bv(E) \to E^{\mathbb N}$, generated by a map $f \colon E \to E$, assigns to each $x \in bv(E)$ the sequence $C_f(x)$ defined by $C_f(x)(n)=f(x(n))$ for $n \in \mathbb N$. Depending on the properties of $f$, the composition operator $C_f$ may map $bv(E)$ into a smaller sequence space. In particular, conditions ensuring that $C_f \colon bv(E) \to l^\infty(E)$ is well-defined were obtained in~\cite{BK}*{Theorem~3.7}. Here, we present a slight improvement of that result.

\begin{proposition}\label{prop:acting_conditions}
The following conditions are equivalent\textup:
\begin{enumerate}[label=\textup{(\roman*)}]
 \item\label{it:i} $C_f(bv(E))\subseteq l^\infty(E)$,
 \item\label{it:ii} $f$ is bounded on compact subsets of $E$,
 \item\label{it:iii} for every point $u \in E$ there exist radii $r,R>0$ such that $f(\ball_E(u,r))\subseteq \ball_E(0,R)$.
\end{enumerate}
\end{proposition}   

\begin{proof}
The equivalence of the first two conditions was shown in~\cite{BK}*{Theorem~3.7}. Now, we will prove the implication $\ref{it:i}\Rightarrow\ref{it:iii}$. Suppose, on the contrary, that there is a point $u \in E$ such that $f(\ball_E(u,\frac{1}{n^2}))\not\subseteq \ball_E(0,n)$ for all $n \in \mathbb N$. Then, there exists a sequence $(w_n)_{n \in \mathbb N}$ such that $\norm{w_n-u}\leq \frac{1}{n^2}$ and $\norm{f(w_n)}>n$ for $n \in \mathbb N$. Set $x:=(w_1,w_2,w_3,\ldots)$. Since $\norm{x(n+1)-x(n)}\leq \norm{w_{n+1}-u}+\norm{u-w_n} \leq \frac{2}{n^2}$, we conclude that $x \in bv(E)$. However, $\norm{f(x(n))}=\norm{f(w_n)}>n$, so $C_f(x)\notin l^\infty(E)$, which is impossible.

The proof of the implication $\ref{it:iii}\Rightarrow\ref{it:i}$ is analogous to that of~\cite{BK}*{Theorem~3.7}. Take $x \in bv(E)$. Since $E$ is a Banach space, the sequence $x$ is convergent to some point $u \in E$. By~\ref{it:iii}, there exist radii $r,R>0$ such that $f(\ball_E(u,r))\subseteq \ball_E(0,R)$. 
But only finitely many terms of $x$ lie outside $\ball_E(u,r)$. Therefore, $C_f(x)\in l^\infty(E)$, proving that $\ref{it:iii}\Rightarrow\ref{it:i}$.
\end{proof}

\begin{remark}
It is worth noting that if $E$ is not complete, then the above characterization fails (see~\cite{BK}*{Example~3.4}).
\end{remark}

\section{Injectivity and surjectivity of $C_f$}

In this section, we study the function-theoretic properties of the composition operator
$C_f$ acting between the spaces $bv(E)$ and $l^\infty(E)$. Since it is defined pointwise,
the characterization of injectivity is straightforward. A simple argument yields the following result.

\begin{proposition}\label{prop:cf_injectivity}
Let $f \colon E \to E$ be a function that is bounded on compact subsets of $E$, and let
$C_f \colon bv(E) \to l^\infty(E)$ be the composition operator generated by $f$. Then, $C_f$ is injective if and only if $f$ is injective.
\end{proposition}

We now turn to the second part of this section, where we characterize the surjectivity of $C_f$.

\begin{proposition}\label{prop:necessary_and_sufficient} 
Let $f \colon E \to E$ be a function that is bounded on compact subsets of $E$, and let
$C_f \colon bv(E) \to l^\infty(E)$ be the composition operator generated by $f$. The following conditions are equivalent\textup:
\begin{enumerate}[label=\textup{(\roman*)}]
 \item\label{it:sur_i} $C_f$ is surjective,

 \item\label{it:sur_ii} for every non-empty, countable and bounded subset $A$ of $E$ there exists a point $u \in E$ such that $A \subseteq f(\ball_E(u,r))$ for any $r>0$,

 \item\label{it:sur_iii} for every non-empty, countable and bounded subset $A$ of $E$ the intersection $\bigcap_{a \in A} \overline{f^{-1}(\{a\})}$ is non-empty.
\end{enumerate}
\end{proposition}

\begin{proof}
Let $A$ be an arbitrary non-empty, countable and bounded subset of $E$, and let $a_1,a_2,\ldots$ be an enumeration of its elements. When $A$ is finite, we regard this enumeration as an infinite sequence by repeating its last element indefinitely. Define $y:=(a_1,a_1,a_2,a_1,a_2,a_3,a_1,a_2,a_3,a_4,\ldots)\in l^\infty(E)$, constructed in such a way that each element $a_n$, for $n \in \mathbb N$, appears infinitely many times in $y$. Since $C_f$ is surjective, there exists $x \in bv(E)$ such that $y=C_f(x)$. Let $u$ denote the limit of $x$; this limit exists, because $bv(E) \subset c(E)$, whenever $E$ is complete. Fix $r>0$. Then, there is an index $N \in \mathbb N$ such that $y(k)=C_f(x)(k) \in f(\ball_E(u,r))$ for all $k \geq N$. Now, for any element $a_n \in A$, we choose $j \geq N$ such that $a_n=y(j)$. It follows that $a_n \in f(\ball_E(u,r))$. Therefore, the implication  $\ref{it:sur_i}\Rightarrow\ref{it:sur_ii}$ holds.

The proof of the implication $\ref{it:sur_ii}\Rightarrow\ref{it:sur_iii}$ is obvious. 

It remains to show $\ref{it:sur_iii}\Rightarrow\ref{it:sur_i}$. Let $y \in l^\infty(E)$ and define $A:=\dset{y(n) \in E}{n \in \mathbb N}$. By~\ref{it:sur_iii} there exists a point $u \in E$ such that for every $n \in \mathbb N$ one can find an element $w_n \in \ball_E(u,\frac{1}{n^2})$ satisfying $y(n)=f(w_n)$. Define $x:=(w_1,w_2,w_3,\ldots)$. Then,
\[
 \norm{x(n+1)-x(n)}\leq \norm{w_{n+1}-u}+\norm{u-w_n} \leq \frac{2}{n^2}
\]
for every $n \in \mathbb N$. Hence, $x \in bv(E)$. Finally, it is immediate from the construction that $y=C_f(x)$.    
\end{proof}

\begin{remark}
Note that conditions~\ref{it:sur_ii} and~\ref{it:sur_iii} in Proposition~\ref{prop:necessary_and_sufficient} are in fact equivalent for arbitrary $A$, and do not require the assumption that $A$ is countable.  
\end{remark}

Every separable Banach space has the \emph{Lindel\"of property}, that is, every open cover of the space admits a countable subcover (see, for example,~\cite{willard}*{Definition~16.5 and Theorem~16.11}). Consequently, for separable Banach spaces, the preceding proposition extends from bounded and countable sets to arbitrary bounded sets.

\begin{proposition}\label{prop:necessary_and_sufficient_sep}
Assume that $E$ is a separable Banach space. Let $f \colon E \to E$ be a function that is bounded on compact subsets of $E$, and let $C_f \colon bv(E) \to l^\infty(E)$ be the composition operator generated by $f$. The following conditions are equivalent\textup:
\begin{enumerate}[label=\textup{(\roman*)}]
 \item\label{it:sur_sep_i} $C_f$ is surjective,

 \item\label{it:sur_sep_ii} for every non-empty and bounded subset $A$ of $E$ there exists a point $u \in E$ such that $A \subseteq f(\ball_E(u,r))$ for any $r>0$,

 \item\label{it:sur_sep_iii} for every non-empty and bounded subset $A$ of $E$ the intersection $\bigcap_{a \in A} \overline{f^{-1}(\{a\})}$ is non-empty.
\end{enumerate}
\end{proposition}

\begin{proof}
Note that it suffices to show that condition~\ref{it:sur_iii} of Proposition~\ref{prop:necessary_and_sufficient} implies~\ref{it:sur_sep_iii} here. Let $A$ be a non-empty and bounded subset of $E$, and for each $a \in A$ set $F_a:=\overline{f^{-1}(\{a\})}$. Suppose that $\bigcap_{a \in A} F_a = \emptyset$. Then, $\bigcup_{a \in A} (E\setminus F_a)=E$. Since $E$ is separable, it has the Lindel\"of property, so there exists a (possibly finite) sequence $(a_n)_{n \in N} \subseteq A$, where $N \subseteq \mathbb N$, such that $\bigcup_{n \in N} (E\setminus F_{a_n})=E$. Equivalently, $\bigcap_{n \in N} F_{a_n}=\emptyset$, which contradicts~\ref{it:sur_iii} of Proposition~\ref{prop:necessary_and_sufficient}. This completes the proof.
\end{proof}

From the results above, we can readily deduce the following properties of the generators of surjective composition operators between $bv(E)$ and $l^\infty(E)$. Except for the first one, they are all rather pathological.

\begin{corollary}\label{cor:properties}
Let $f \colon E \to E$ be a function that is bounded on compact subsets of $E$, and let $C_f \colon bv(E) \to l^\infty(E)$ be the surjective composition operator generated by $f$. Then,
\begin{enumerate}[label=\textup{(\roman*)}]
 \item $f$ is surjective,

 \item\label{it:cor_prop_ii} for any distinct $u,w \in E$ the intersection $\overline{f^{-1}(\{u\})} \cap \overline{f^{-1}(\{w\})}$ is non-empty, and $f$ is discontinuous at every point of this intersection,

\item $\inf_{a\in f^{-1}(\{u\})} \inf_{b\in f^{-1}(\{w\})} \norm{a-b}=0$ for all $u,w \in E$,

 \item for every $u \in E$ and $r>0$ we have $f\bigl( \bigcup_{a \in f^{-1}(\{u\})} \ball_E(a,r)\bigr) =E$.
\end{enumerate}
\end{corollary}

Finally, combining the results on injectivity and bijectivity of the composition operator, we obtain the following rather surprising fact, all the more striking given that $bv(E)$ and $l^\infty(E)$ have the same cardinality.

\begin{corollary}
Let $f \colon E \to E$ be a map that is bounded on compact subsets of $E$. No composition operator $C_f \colon bv(E) \to l^\infty(E)$ is bijective. 
\end{corollary}

\section{Examples}
A natural question is whether there exist maps $f$ that generate surjective composition operators between $bv(E)$ and $l^\infty(E)$. The answer is positive: such maps do indeed exist. In what follows, we present a concrete example of one such mapping.

\begin{example}\label{eq:example_PK}
Let $E$ be an infinite-dimensional real Banach space, and let $H$ be its Hamel basis. Without loss of generality we may assume that every element of the basis has norm $1$. Set
\[
 Q:=\dset[\Bigg]{\sum_{i=1}^m q_i h_i}{ \text{$q_i \in \mathbb Q$ and $h_i \in H$ for $i=1,\ldots,m$, where $m \in \mathbb N$ is arbitrary}}.
\]
The relation $\sim$ defined for any two elements $u,w \in E$ by
\[
 \text{$u \sim w$ if and only if $u-w \in Q$}
\]
is an equivalence relation. By $[u]$ let us denote the equivalence class of $u \in E$, that is, the set of all points $w \in E$ such that $u \sim w$.

Now, we will show that each equivalence class is dense in $E$, that is, for every radius $r>0$ and all points $u,w \in E$ the intersection $B_E(u,r) \cap [w]$ is non-empty. Let $h_1,\ldots,h_m \in H$ and $p_1,\ldots,p_m \in \mathbb R$ be such that $u-w=\sum_{i=1}^m p_i h_i$. Choose $q_1,\ldots,q_m \in \mathbb Q$ so that $\abs{p_i-q_i}< \frac{r}{m}$ for $i=1,\ldots,m$. Then, $v:=w+\sum_{i=1}^m q_i h_i \in [w]$ and
\[
 \norm{u-v}\leq \sum_{i=1}^m \abs{p_i-q_i}\cdot \norm{h_i} = \sum_{i=1}^m \abs{p_i-q_i} <r.
\] 
Thus, $v \in B_E(u,r)$, proving that $B_E(u,r) \cap [w] \neq \emptyset$. In particular, for every $n \in \mathbb N$ and $u \in E$ the set $[u]_n :=[u] \cap \bigl( B_E(0,n)\setminus \ball_E(0,n-1)\bigr)$ is dense in $B_E(0,n)\setminus \ball_E(0,n-1)$. 

Fix $n \in \mathbb N$, and let $A_n := \dset{[u]_n}{u \in E}$.  It is clear that $\card A_n \leq \card \ball_E(0,n) \leq \card E$. To show that the sets $A_n$, $\ball_E(0,n)$ and $E$ have the same cardinality, choose a non-rational number $p$ and consider a map $\psi \colon H \to A_n$ defined by $\psi(h):=[ph]_n$. The map $\psi$ is an injection, as if $g$ and $h$ are two distinct elements of the Hamel basis, then $[p g]\cap [p h] = \emptyset$. Indeed, suppose, to the contrary, that 
\[
 pg-ph = \sum_{i=1}^m q_i h_i
\]
for some $q_i \in \mathbb Q$ and $h_i \in H$. Without loss of generality, we may assume that $g,h$ are among $h_1,\ldots,h_m$; say, $g=h_1$ and $h=h_2$. Then, $(p-q_1)h_1 + (-p-q_2)h_2 + q_3h_3+\ldots+q_m h_m=0$. Since $p-q_1\neq 0$ and $p+q_2 \neq 0$ this implies that the vectors $h_1,\ldots,h_m$ are not linearly independent -- a contradiction. Consequently, $\card H \leq \card A_n$. By~\cite{HH}*{Theorem~3.5} we also know that $\card E = \card H$. The Cantor--Bernstein theorem, along with the previous inequalities, therefore yields $\card A_n = \card \ball_E(0,n)$ for every $n \in \mathbb N$ (see~\cite{willard}*{Problem~1J}). Hence, there exists a bijection $\varphi_n \colon A_n \to \ball_E(0,n)$. Additionally, as $E$ is an infinite-dimensional Banach space, by Riesz's lemma it admits a sequence $(e_k)_{k \in \mathbb N}$ of unit vectors such that $\norm{e_k - e_l}\geq \frac{1}{2}$ for $k\neq l$ (cf.~\cite{kreyszig}*{Theorem~2.5-4}).  

We are ready to define the function $f \colon E \to E$. Set
\[
 f(u):=\begin{cases}
         ke_k, & \text{if $u=e_k$ for some $k \in \mathbb N$,}\\
				 \varphi_n\bigl([u]_n\bigr), & \text{if $u \in B_E(0,n)\setminus \ball_E(0,n-1)$ for some $n \in \mathbb N$,}\\
				 0, & \text{otherwise}.
      \end{cases}
\]
Note that $f$ is bounded on compact subsets of $E$, because any compact subset contains at most finitely many vectors $e_k$. Consequently, $f$ generates the composition operator $C_f$ which acts between $bv(E)$ and $l^\infty(E)$. Furthermore, due to the density of each equivalence class in $E$, the map $f$ satisfies also condition~\ref{it:sur_ii} of Proposition~\ref{prop:necessary_and_sufficient}, meaning that $C_f$ is a surjection. Note also that in view of Corollary~\ref{cor:properties}~\ref{it:cor_prop_ii} the generator $f$ is discontinuous at every point of $E$.
\end{example}

\begin{remark}
The vectors $e_k$ are not essential for the definition of $f$. They are introduced only to illustrate that $f$ can be made bounded on compact subsets of $E$, yet unbounded on bounded ones.
\end{remark}

\begin{remark}
The above example works in infinite-dimensional spaces. However, it can be modified slightly to work in finite-dimensional spaces as well. We begin with $E:=\mathbb R$. We also take $H:=\{1\}$ and $Q:=\mathbb Q$. Since each equivalence class $[u]$ is merely a translation of the set $\mathbb Q$, it follows that the cardinality of every family $A_n$ must be continuum. Therefore, $\card A_n =\card \ball_{\mathbb R}(0,n)$ for $n \in \mathbb N$. Now, it suffices to define the map $f$ using a similar formula to the one used in Example~\ref{eq:example_PK}, but without the part involving the sequence of unit vectors.    

In higher dimensions, we use the analogous construction along a given line in $\mathbb R^N$. Outside of this line, we define the map to be zero.
\end{remark}

The general example of a map satisfying condition~\ref{it:sur_ii} of Proposition~\ref{prop:necessary_and_sufficient} and being bounded on compact sets is not constructive due to the nature of the bijections $\varphi_n$. However, in the case $E:=\mathbb R$, a simpler example can be provided.

\begin{example}\label{ex:example_PM}
Let $(w_j)_{j \in \mathbb N}$ be a sequence of positive numbers that sum up to $1$. Also, let $s_k:=w_1+\ldots+w_k$ for $k \in \mathbb N$. Define the function $f \colon \mathbb R \to \mathbb R$ by the following formula
\[
 f(u):=\begin{cases}
       -n, &  \text{if $u=n+s_{2k}$ for $n,k \in \mathbb N$,}\\
			  n, &  \text{if $u=n+s_{2k+1}$ for $n,k \in \mathbb N$,}\\
				\text{linear}, & \text{on the interval $[n+s_{2k},n+s_{2k+1}]$, where $n,k \in \mathbb N$},\\
				0, & \text{otherwise}.
			\end{cases}
\]
Clearly, the function $f$ is bounded on compact subsets of $\mathbb R$, as $f([-(n+1),(n+1)]) = [-n,n]$ for $n \in \mathbb N$. Additionally, it satisfies condition~\ref{it:sur_ii} of Proposition~\ref{prop:necessary_and_sufficient} with $u$ being a sufficiently large positive integer.
\end{example}

We conclude the paper with two further examples showing that, in a certain sense, Propositions~\ref{prop:necessary_and_sufficient} and~\ref{prop:necessary_and_sufficient_sep} are optimal. First, we show that in Proposition~\ref{prop:necessary_and_sufficient} the countability requirement in condition~\ref{it:sur_iii} cannot, in general, be replaced by the weaker requirement that the corresponding intersection be non-empty for every non-empty finite subset $A$. Second, we show that the separability of $E$ in Proposition~\ref{prop:necessary_and_sufficient_sep} cannot, in general, be dispensed with.

\begin{example}
Let $A:=\dset{\frac{1}{m}}{m \geq 2}$. Since $A$ is countable, the family of all its non-empty finite subsets is countable as well. Let $A_1,A_2,\ldots$ be an enumeration of these subsets.

Our goal is to construct a function $g\colon \mathbb R\to\mathbb R$ such that for every finite set $B \subset \mathbb R$ the intersection $\bigcap_{b\in B}\overline{g^{-1}(\{b\})}$ is non-empty, whereas for the whole set $A$ we have $\bigcap_{a\in A}\overline{g^{-1}(\{a\})}=\emptyset$. To achieve this, we slightly modify the function $f$ introduced in Example~\ref{ex:example_PM}. For each interval $[n+s_{2k},n+s_{2k+1}]$, where $n,k\in\mathbb N$, we proceed as in the construction of $f$, connecting the points $(n+s_{2k},-n)$ and $(n+s_{2k+1},n)$ by a line segment. This time, however, we introduce ``holes'' in the graph by setting $g(u):=0$ whenever $u\in (n,n+1) \cap f^{-1}(A\setminus A_n)$. Then, the required properties of $g$ can be deduced easily from the construction.

Since $g$ is bounded on bounded subsets of $\mathbb R$, it generates a composition operator $C_g\colon bv(\mathbb R)\to l^\infty(\mathbb R)$. However, $C_g$ is not surjective. Indeed, consider the bounded sequence $y:=(\frac12,\frac13,\ldots)$. By the construction of $g$, every sequence $x$ satisfying $C_g(x)=y$ must be unbounded. Consequently, it cannot belong to $bv(\mathbb R)$, since $bv(\mathbb R)\subset l^\infty(\mathbb R)$.
\end{example}

\begin{example}
Consider $E:=l^\infty(\mathbb R)$, and let $H$ be its Hamel basis. Without
loss of generality, we may assume that every element of the basis has norm $1$. We also consider the equivalence relation introduced in Example~\ref{eq:example_PK}. Additionally, let $S$ denote the subset of $E$ consisting of all zero-one sequences, and let $P$ be the family of all non-empty countable subsets of $S$. By standard cardinal arithmetic, each of the sets $E$, $H$, $S$, and $P$ has cardinality continuum (cf. also~\cite{HH}*{Theorem~3.5}). Hence, we may index the elements of $P$ by elements of $S$, writing $P=\dset{P_s}{s \in S}$.

Arguing as in Example~\ref{eq:example_PK}, one can show that for every $n\in\mathbb N$ and $u\in E$, the set
\[
 [u]_n := [u] \cap \biggl( B_E(0,n) \setminus \biggl( \ball_E(0,n-1) \cup \bigcup_{s \in S} \ball_{E}\bigl((n-\tfrac{1}{2})s ,\tfrac{1}{8}\bigr)\biggr)\biggr)
\]
is dense in
\[
B_E(0,n) \setminus \biggl( \ball_E(0,n-1) \cup \bigcup_{s \in S} \ball_{E}\bigl((n-\tfrac{1}{2})s ,\tfrac{1}{8}\bigr)\biggr),
\]
and that the family $A_n:=\dset{[u]_n}{u \in E}$ has cardinality continuum. Since $\ball_E(0,n)\setminus S$ also has cardinality continuum there exists a bijection $\varphi_n \colon A_n \to \ball_E(0,n)\setminus S$.

Similarly, for every $n\in\mathbb N$, $s\in S$, and $u\in E$, the set
\[
 \langle u \rangle_{n,s}:= [u] \cap \ball_{E}\bigl((n-\tfrac{1}{2})s ,\tfrac{1}{8}\bigr)
\]
is dense in $\ball_{E}\bigl((n-\tfrac{1}{2})s ,\tfrac{1}{8}\bigr)$. Let $B_{n,s}:=\dset{\langle u \rangle_{n,s}}{u \in E}$. Again, $B_{n,s}$ has cardinality continuum. Therefore, we may fix a surjection $\psi_{n,s}$ of $B_{n,s}$ onto $P_s$, as $P_s$ is countable.

We define the function $f \colon E \to E$ by
\[
 f(u):=\begin{cases}
         ke_k, & \text{if $u$ is the $k$-th unit vector $e_k$ for some $k \in \mathbb N$,}\\
				 \varphi_n\bigl([u]_n\bigr), & \text{if $\displaystyle u \in B_E(0,n) \setminus \biggl( \ball_E(0,n-1) \cup \bigcup_{s \in S} \ball_{E}\bigl((n-\tfrac{1}{2})s ,\tfrac{1}{8}\bigr)\biggr)$ for some $n \in \mathbb N$,}\\
				 \psi_{n,s}\bigl(\langle u \rangle_{n,s}\bigr), & \text{if $\displaystyle u \in \ball_{E}\bigl((n-\tfrac{1}{2})s ,\tfrac{1}{8}\bigr)$ for some $n \in \mathbb N$ and $s \in S$,}\\
				 0, & \text{otherwise}.
      \end{cases}
\]
Note that $f$ is indeed well-defined, since the balls $\ball_{E}\bigl((n-\tfrac{1}{2})s ,\tfrac{1}{8}\bigr)$ are disjoint for different $n \in \mathbb N$ and $s \in S$. By construction, for every non-empty, countable and bounded subset $C$ of $E$ the intersection $\bigcap_{c \in C} \overline{f^{-1}(\{c\})}$ is non-empty, whereas 
\[
  \bigcap_{s\in S} \overline{f^{-1}(\{s\})} = \emptyset.
\]
\end{example}

\subsection*{Use of Large Language Models}
All mathematical results, proofs, and arguments presented in this work were independently derived, fully formalized, and rigorously verified by the authors. No statement in this paper was accepted on the basis of model output alone, and the LLM was not used as a source of validated mathematical claims. Its role was limited to informal assistance in discussion and writing.

\end{document}